\documentclass[a4paper,11pt]{amsart}
\usepackage{amssymb,amsmath,amsthm}
\usepackage[normalem]{ulem}
\usepackage{fullpage}
\usepackage{color}
\definecolor{darkblue}{rgb}{0.2,0.2,0.6}
\definecolor{darkred}{rgb}{0.6,0.1,0.1}
\usepackage[colorlinks=true,linkcolor=darkblue, citecolor=darkred]{hyperref}
\usepackage{eucal, lmodern}
\usepackage{stackrel, enumerate}
\usepackage{graphicx, bbm, tcolorbox}
\DeclareMathOperator*{\essup}{ess\,sup}
\DeclareMathOperator*{\esinf}{ess\,inf}
\newcommand{\ie}{\emph{i.e.}}
\newcommand{\eg}{\emph{e.g.}}

\newcommand{\cf}{\emph{cf.}}

\newcommand\nb{\nabla}
\newcommand{\beq}{\begin{equation} \begin{split}}
\newcommand{\eeq}{\end{split} \end{equation}}

\newcommand\Omg{\Omega}

\renewcommand\and{\qquad\text{and}\qquad}

\newcommand\sm{\setminus}

\newcommand{\comm}[1]{}

\def\sfH{\mathsf{H}}

\def\bm1{\mathbbm{1}}

\def\p{\partial}

\def\arr{\rightarrow}

\def\tt{\theta}
\def\aa{\alpha}

\def\ii{{\mathsf{i}}}
\def\p{\partial}

\def\sfH{\mathsf{H}}

\def\dd{{\,\mathrm{d}}}
\def\der{{\mathrm{d}}}

\newcounter{counter_a}
\newenvironment{myenum}{\begin{list}{{\rm(\roman{counter_a})}}%
{\usecounter{counter_a}
\setlength{\itemsep}{1.ex}\setlength{\topsep}{0.8ex}
\setlength{\leftmargin}{5ex}\setlength{\labelwidth}{5ex}}}{\end{list}}

\usepackage[latin1]{inputenc}
\usepackage[T1]{fontenc}

\numberwithin{figure}{section}
\numberwithin{equation}{section}
\theoremstyle{plain}
\newtheorem*{thm*}{Theorem}
\newtheorem{thm}{Theorem}[section]

\newtheorem{lem}[thm]{Lemma}
\newtheorem{prop}[thm]{Proposition}

\newtheorem{cor}[thm]{Corollary}
\newtheorem{conj}[thm]{Conjecture}

\theoremstyle{remark}
\newtheorem{remark}[thm]{Remark}

\theoremstyle{plain}

\newcommand{\beu}{\begin{equation*}}
\newcommand{\eeu}{\end{equation*}}
\newcommand{\besu}{\begin{equation*}
\begin{aligned}}
\newcommand{\eesu}{\end{aligned}
\end{equation*}}
\newcommand{\bes}{\begin{equation}
\begin{aligned}}
\newcommand{\ees}{\end{aligned}
\end{equation}}

\newcommand\cA{\mathcal A}
\newcommand\cB{\mathcal B}

\newcommand\cH{\mathcal H}

\newcommand\frq{\mathfrak q}

\newcommand\frh{\mathfrak h}

\newcommand\ov{\overline}
\newcommand\wt{\widetilde}
\newcommand\wh{\widehat}

\DeclareMathOperator\dom{dom}
\DeclareMathOperator\ran{ran}

\newcommand\void[1]{}

\def\ov{\overline}
\def\eps{\varepsilon}
\def\ran{{\rm ran\,}}

\def\frt{{\mathfrak t}}

      \def\dC{{\mathbb C}}

      \def\dR{{\mathbb R}}

   \def\dZ{{\mathbb Z}}

   \def\sfH{{\mathsf H}}

   \def\sfT{{\mathsf T}}

\def\cA{{\mathcal A}}   \def\cB{{\mathcal B}}   
   \def\cE{{\mathcal E}}   
   \def\cH{{\mathcal H}}   
      
      \def\cO{{\mathcal O}}
      
      \def\cU{{\mathcal U}}

\renewcommand{\div}{\mathrm{div}\,}

\newcommand{\Tr}{\mathrm{Tr}\,}

\usepackage[left=2.7cm, right=2.7cm, marginparwidth=2cm, textheight = 24cm ]{geometry}
\usepackage[normalem]{ulem}
\definecolor{DarkGreen}{rgb}{0,0.5,0.1}

\definecolor{DarkBlue}{rgb}{0,0.1,0.5}

\newcommand\soutD{\bgroup\markoverwith
	{\textcolor{DarkGreen}{\rule[.5ex]{2pt}{1pt}}}\ULon}
\newcommand\soutP{\bgroup\markoverwith
	{\textcolor{blue}{\rule[.5ex]{2pt}{1pt}}}\ULon}
\newcommand{\Hm}[1]{\leavevmode{\marginpar{\tiny%
			$\hbox to 0mm{\hspace*{-0.5mm}$\leftarrow$\hss}%
			\vcenter{\vrule depth 0.1mm height 0.1mm width \the\marginparwidth}%
			\hbox to
			0mm{\hss$\rightarrow$\hspace*{-0.5mm}}$\\
			\relax\raggedright #1}}}

\newcommand\essinf{{\rm ess\,inf}\,}
\newcommand\esssup{{\rm ess\,sup}\,}

\newcommand{\bA}{\mathbf{A}}

\theoremstyle{remark}
\newtheorem{exam}[thm]{Example}

\begin{document}

\title[A geometric bound via the torsion function]{A geometric bound on the lowest magnetic Neumann eigenvalue via the torsion function}

\author[A.~Kachmar]{Ayman Kachmar}
\address{(A.~Kachmar) 
The Chinese University of Hong Kong (Shenzhen),  School of Science and Engineering,  Shenzhen 518172,  
China}
\email{akachmar@cuhk.edu.cn}
\author[V.~Lotoreichik]{Vladimir Lotoreichik}
\address{(V.~Lotoreichik)
	Department of Theoretical Physics\\
	Nuclear Physics Institute, Czech Academy of Sciences, 
	25068, \v{R}e\v{z}, Czech Republic
}
\email{lotoreichik@ujf.cas.cz}

\begin{abstract}
We obtain an upper bound on the lowest magnetic Neumann eigenvalue of a bounded, convex, smooth, planar domain
with moderate intensity of the homogeneous magnetic field. 
This bound is given as a product of a purely geometric factor expressed in terms of the torsion function and of the lowest magnetic Neumann eigenvalue of the disk having the same maximal value of the torsion function as the domain.
The bound is sharp in the sense that equality is attained for disks. 
Furthermore, we derive from our upper bound that the lowest magnetic Neumann eigenvalue with the homogeneous magnetic field is maximized by the disk among all ellipses of fixed area provided that the intensity of the magnetic field does not exceed an explicit constant dependent only on the fixed area.  
\end{abstract}


\maketitle

\section{Introduction}
\subsection{Motivation and background}
Geometric bounds on the eigenvalues are one of the central topics in spectral geometry with many contributions and challenging open problems. A prominent class of such bounds constitute spectral isoperimetric inequalities, which can be viewed as optimization of an eigenvalue under certain geometric constraints.
In the seminal papers by
Faber~\cite{F} and Krahn~\cite{K, K2}, it was proved that among all bounded domains of fixed volume the ball minimizes the lowest Dirichlet eigenvalue of the Laplacian. For the Neumann Laplacian the lowest eigenvalue on a bounded domain is always equal to zero. The natural question is thus to optimize the first non-zero eigenvalue. It was shown by Szeg\H{o}~\cite{S54} in two dimensions and soon after by Weinberger~\cite{W56} in higher dimensions that the first non-zero Neumann eigenvalue of the Laplacian is maximized by the ball among all bounded domains of fixed volume. Other types of geometric bounds on the eigenvalues involve also purely geometric characteristics of the domain such as volume, area of the boundary, in-radius, diameter, and others; see \eg~\cite{FK08, PS51}
for upper bounds on the lowest Dirichlet eigenvalue  and \eg~\cite{B03, PW60} for a lower bound on the first non-zero Neumann eigenvalue.

Geometric bounds on the eigenvalues and spectral optimization for the
two-dimensional magnetic Laplacian  attracted considerable attention in the recent years. L.~Erd\H{o}s~\cite{E} proved
an analogue of the Faber-Krahn inequality for the two-dimensional magnetic Laplacian with Dirichlet boundary conditions and homogeneous magnetic field. A quantitative version of the inequality by Erd\H{o}s has been recently obtained in~\cite{GJM23} by Ghanta, Junge, and Morin.  Much less is known beyond Dirichlet boundary conditions.  For the magnetic Laplacian, the lowest Neumann eigenvalue is different from zero and depends on the shape of the domain. Its optimization is thus a meaningful problem unlike in the non-magnetic case.  Under an Aharonov-Bohm magnetic potential,  fixing the area,  it is proved in \cite{CPS} that the lowest magnetic Neumann eigenvalue is maximized by the disk with the pole of the Aharonov-Bohm potential at its center. 
The same optimization problem under a homogeneous magnetic field is more subtle. Based on the analysis of the asymptotic expansions of the lowest magnetic Neumann eigenvalue in the regimes of weak and strong magnetic fields, Fournais and Helffer conjectured in~\cite{FH} that among all simply-connected planar domains of fixed area the disk is a maximizer of the lowest magnetic Neumann eigenvalue for the homogeneous magnetic field. An optimization result of a different flavour for the counting function of the magnetic Neumann Laplacian at the points of the Landau levels
has been obtained in~\cite{FFGKS23}. Various geometric bounds on the lowest magnetic Neumann eigenvalue
in two dimensions are derived in~\cite{CLPS23}. 

Along with Dirichlet and Neumann boundary conditions, the Robin boundary conditions for the magnetic Laplacian are frequently studied.
The authors of the present paper considered in~\cite{KL22} optimization for the two-dimensional magnetic Robin Laplacian with a negative boundary parameter.  In particular, we proved in that paper that the disk is a maximizer of the lowest magnetic Robin eigenvalue among all convex centrally symmetric domains under fixed perimeter constraint and certain restrictions on the intensity of the homogeneous magnetic field.   The convexity assumption is removed in a subsequent paper~\cite{DKL23} in collaboration with Dietze.

The analysis in the present paper is inspired by the
above-mentioned conjecture due to Fournais and Helffer on the optimization of the lowest magnetic Neumann eigenvalue in two dimensions. We obtain an upper bound on the lowest magnetic Neumann eigenvalue with homogeneous magnetic field for planar convex domains in terms of the torsion function under the assumption that the intensity of the magnetic field does not exceed an explicit constant dependent on the area of the domain.
By testing the bound on ellipses, for which the torsion function is known explicitly, we get
under our restriction on the intensity of the magnetic field an inequality, which
is stronger than Fournais-Helffer conjecture in this special geometric setting. We expect that the obtained upper bound implies Fournais-Helffer conjecture for a larger class of convex domains. We have also found examples of planar convex domains, for which our bound gives an estimate, which is weaker than the Fournais-Helffer conjecture.  

\subsection{Magnetic Neumann Laplacian}
Let $\Omg\subset\dR^2$ be a bounded, simply-connected $C^\infty$-smooth domain. Let the intensity of the magnetic field $b >0$ be fixed. Let us also fix the gauge of the homogeneous magnetic field by
\[
	\wt\bA(x) := \frac{b}{2}\left(-x_2,x_1\right)^\top.
\]
Consider the symmetric quadratic form in the Hilbert space $L^2(\Omg)$
\[
	H^1(\Omg)\ni u \mapsto \int_{\Omg} |(-\ii\nb -\wt\bA)u|^2\dd x.
\]
This form is closed, densely defined, and non-negative 
(see~\cite[\S 1.2]{FH-b}). Thus, it defines a self-adjoint operator in the Hilbert space $L^2(\Omg)$, which we regard as the magnetic Neumann Laplacian on $\Omg$ with the homogeneous magnetic field of intensity $b$. In our analysis we work with a different gauge of the magnetic field, but for the purpose of the introduction we prefer to define the operator in a simpler gauge. Recall that the spectrum of the magnetic Neumann Laplacian on a simply-connected planar domain is invariant under the change of the gauge; \cf~\cite[Appendix D]{FH-b} and also the discussion in~\cite[Section 2]{KL22}.

The study of the magnetic Neumann Laplacian is motivated by applications in the theory of superconductivity; see the monograph~\cite{FH-b} and the references therein.
The spectrum of the magnetic Neumann Laplacian on $\Omg$ is purely discrete thanks to compactness of the embedding of $H^1(\Omg)$ into $L^2(\Omg)$ and we denote by $\mu_1^b(\Omg) > 0$ its lowest eigenvalue. We are interested in geometric bounds on $\mu_1^b(\Omg)$ and its optimization with respect to the shape of $\Omg$.
As we already mentioned, our analysis is inspired by the conjecture
due to Fournais and Helffer~\cite[Equation (1.8)]{FH}, which we state below in full detail. 
\begin{conj}\label{conj:FH}
	Let $\Omg\subset\dR^2$ be a bounded, simply-connected
	$C^\infty$-smooth domain and let $\cB\subset\dR^2$ be the disk of the same area as $\Omg$. Then the following inequality
	\begin{equation*}
		\mu_1^b(\Omg)\le \mu_1^b(\cB)
	\end{equation*}
	holds for all $b > 0$.
\end{conj}
\subsection{On the main results}
Under the assumption $b|\Omg| < \pi$ we obtain in Theorem~\ref{thm:main} below an upper bound on $\mu_1^b(\Omg)$ for convex $\Omg$. In the following, we denote by $R > 0$ the radius of the disk $\cB\subset\dR^2$ of the same area as $\Omg$. Our bound can be written as
\begin{equation}\label{eq:equivalent_bound}
	\mu_1^b(\Omg) \le C(\Omega)\mu_1^b(\cB_{\rho(\Omg)}), 
\end{equation}
where the radius $\rho(\Omg) \le R$ of the disk
$\cB_{\rho(\Omg)}\subset\dR^2$ and the constant \(C(\Omega) \ge1\) depend on the domain \(\Omega\) only.
We remark that the bound in~\eqref{eq:equivalent_bound} is sharp in the sense that $\rho(\Omg) = R$ and $C(\Omg) = 1$ when \(\Omega\) is a disk. The radius $\rho(\Omg)$ is
such that the maximum of the torsion function for $\cB_{\rho(\Omg)}$ coincides with that for $\Omg$. Thus, $\rho(\Omg)$ can be explicitly expressed through the maximal value of the torsion function for $\Omg$. The expression for $C(\Omg)$ involves certain integrals of the  modulus of the gradient and of its inverse of the torsion function on its level sets, which
can be expressed in terms of the areas of superlevel sets of the torsion function for $\Omg$.
The bound~\eqref{eq:equivalent_bound} is obtained by constructing a trial function whose level sets are the same as those of the torsion function.

The assumption $b|\Omg| < \pi$ ensures that to the lowest eigenvalue of the magnetic Neumann Laplacian on the disk $\cB_{\rho(\Omg)}$ corresponds a radial real-valued eigenfunction. The assumption $b|\Omg| < \pi$ is not optimal and can be replaced by a more relaxed assumption, which guarantees the same property of the ground state for the disk of radius $\rho(\Omg)$. 

We also establish a new monotonicity property of the lowest eigenvalue for the disk, which was conjectured in \cite{CLPS23} based on numerical evidence. Using this property we infer that $\mu_1^b(\cB_{\rho(\Omg)}) \le \mu_1^b(\cB)$ and we get as a consequence of~\eqref{eq:equivalent_bound}
the following simplified bound
\[
	\mu_1^b(\Omg) \le C(\Omg)\mu_1^b(\cB),
\]
where as before $\cB\subset\dR^2$ is the disk of the same area as $\Omg$.

For certain special domains such as ellipses the torsion function is known explicitly. 
Let us define the ellipse with semi-axes $\aa,\beta >0$  by
\[
	\cE_{\aa,\beta} := \left\{(x,y)\in\dR^2\colon \frac{x^2}{\aa^2}+\frac{y^2}{\beta^2} < 1\right\}.
\]
Assume also that $\cE_{\aa,\beta}$ is not a disk, \ie~ $\alpha\not=\beta$. We verify that for the ellipse $\cE_{\aa,\beta}$ we have in~\eqref{eq:equivalent_bound} the following values for the constants
\[
C(\cE_{\aa,\beta}) = 1\quad\text{and}\quad \rho(\cE_{\aa,\beta}) = \frac{\sqrt{2}\aa\beta}{\sqrt{\aa^2+\beta^2}} < \sqrt{\aa\beta}.
\]
As a consequence of our general upper bound  we prove under the assumption $b\aa\beta < 1$ the following inequalities  
\begin{equation}\label{eq:ellipse_intro}
	\mu_1^b(\cE_{\aa,\beta})\le \mu_1^b(\cB_{\rho(\cE_{\aa,\beta})}) < \mu_1^b(\cB),
\end{equation}
where $\cB\subset\dR^2$ is the disk such that $|\cB| = |\cE_{\aa,\beta}| = \pi\aa\beta$. 
In other words, we have shown that among all ellipses of fixed area the disk is the unique maximizer of the lowest magnetic Neumann eigenvalue provided that the intensity of the magnetic field is below an explicit constant, which only depends on the fixed area.

We expect that the upper bound~\eqref{eq:equivalent_bound} yields the inequality in Conjecture~\ref{conj:FH} under the assumption $b|\Omg| < \pi$ for a class of domains much wider than ellipses. Informally speaking, the fact that in some cases $C(\Omg) > 1$  can be compensated by the fact that $\rho(\Omg) < R$.  Slight deformations of ellipses should be allowed by continuity.  
Derivation of Conjecture~\ref{conj:FH} from the inequality~\eqref{eq:equivalent_bound} for a class of convex domains other than ellipses would require more precise estimates for the torsion function. We have also identified a class of planar convex domains for which the bound in~\eqref{eq:equivalent_bound} is weaker than 
the bound in Conjecture~\ref{conj:FH} for all sufficiently small $b > 0$. A typical example of such a `bad' domain is given by $\{(x,y)\in\dR^2\colon x^4 + y^4 < 1\}$ (see Example~\ref{ex:square}).

\subsection*{Structure of the paper}
In Section~\ref{sec:prelim} we provide preliminary material on the torsion function  and the magnetic Neumann Laplacian, which will be used in the formulations and the proofs of the main results.  In Section~\ref{sec:bound} we obtain an upper bound on the lowest magnetic Neumann eigenvalue for planar, convex, smooth domains with moderate intensity of the magnetic field.   In Section~\ref{sec:geom-cst} we study the geometric quantities that appear in our new eigenvalue upper bound.  
Finally, Section~\ref{sec:FH} is devoted to the proof of the Fournais-Helffer conjecture for ellipses.
\section{Preliminaries}\label{sec:prelim}
\subsection{Torsion function of a general domain}\label{ssec:torsion}
The aim of this subsection is to collect properties of the torsion function of a planar domain, which will be used in the proof of our main results. The literature on the torsion function is quite extensive; see
the papers~\cite{B20, BBK22, BMS22, HNPT18, HS21,  K85, KL87, M-L71, S18}, the monograph~\cite{Sp81}, and the references therein.

Let $\Omg\subset\dR^2$ be a bounded, simply-connected, $C^\infty$-smooth domain. We denote by $\p\Omg$ the boundary of $\Omg$ and by $\nu$ the outer unit normal to $\Omg$. 
For the domain $\Omg$, we define the torsion function $v_\Omg \colon\Omg\arr\dR$ as the unique solution
in the Sobolev space $H^1_0(\Omg)$ of the boundary value problem 
\begin{equation}\label{eq:torsion}
	\begin{cases}
		-\Delta v_\Omg = 1,&\text{in}\,\,\Omg,\\
		v_\Omg = 0, &\text{on}\,\,\p\Omg.
	\end{cases}	
\end{equation}
By standard elliptic regularity theory we infer that $v_\Omg\in C^\infty(\ov\Omg)$.
The function $v_\Omg$ is positive in $\Omg$ by the maximum principle for subharmonic functions~\cite[Theorem 8.1]{GT} applied to $-v_\Omg$. Talenti's comparison principle~\cite[Theorem 3.1.1]{Ke} yields that
\begin{equation}\label{eq:max_torsion}
	\max v_\Omg \le \max v_\cB = \frac{R^2}{4},
\end{equation}
where $\cB = \cB_R\subset\dR^2$ is the disk of the same area as $\Omg$. Here we used that the torsion function for the disk $\cB_R$ is explicitly given in polar coordinates by $v_\cB(r) = \frac{R^2-r^2}{4}$. The inequality in~\eqref{eq:max_torsion} is strict if $\Omg$ is not congruent to $\cB$ (see~\cite[Propositions 3.2.1 and 3.2.2]{Ke}).
For the torsion function $v_\Omg$ we introduce the level sets 
\begin{equation}\label{eq:Ar}
	\cA_r := \left\{x\in\Omg\colon v_\Omg(x) = \max v_\Omg\Big(1-r^2\Big)\right\},\qquad r\in(0,1).
\end{equation}
The selected normalization of the level sets for $v_\Omg$ is adjusted for our application to the analysis of the magnetic Neumann Laplacian. We also define the superlevel sets 
\begin{equation}\label{eq:Omgr}
\Omg_r := \left\{x\in\Omg\colon v_\Omg(x) >\max v_\Omg \big(1-r^2\big)\right\},\qquad r\in(0,1).
\end{equation}
In the rest of this subsection we assume in addition that $\Omg$ is convex.
In the next proposition we collect known properties of the torsion function for a planar convex smooth domain.
\begin{prop}\label{prop:convex}
	Let $\Omg\subset\dR^2$ be a bounded, convex, $C^\infty$-smooth domain. Let $v_\Omg\in C^\infty(\ov\Omg)$ be the torsion function of $\Omg$ defined as the unique solution of the boundary value problem~\eqref{eq:torsion}. Then the following hold.
	\begin{myenum}
		\item $\sqrt{v_\Omg}$ is strictly concave. In particular, $v_\Omg$ has a unique global maximum
		at some point $x_\star\in\Omg$, the level sets of $v_\Omg$ defined in~\eqref{eq:Ar} are boundaries of convex domains, and $\nb v_\Omg(x)\ne 0$ for all $x\in\Omg\sm\{x_\star\}$.
		\item The Hessian $\cH_\star\in\dR^{2\times 2}$ of $v_\Omg$ at the point of its maximum is negative definite.
		\item $|\nb v_\Omg(x)|^2 \le 
		2(\max v_\Omg - v_\Omg(x))$ for any $x\in\Omg$.
		\item $\max v_\Omg - v_\Omg(x) \le \frac12|x-x_\star|^2$ for any $x\in\Omg$.
	\end{myenum}
\end{prop}
Concavity of $\sqrt{v_\Omg}$ 
is proved in~\cite{M-L71, K85}. This property is improved in~\cite[Theorem~1]{KL87} to strict concavity stated in item (i) of the above proposition. Negative definiteness of the Hessian
in item (ii) is shown \eg~in~\cite[Theorem 1]{S18}. The estimate of the gradient of the torsion function in item~(iii) and the bound on the torsion function itself in item~(iv) can be found in~\cite[\S 6.1]{Sp81}.

Using properties of the torsion function for planar convex domains we derive an estimate on the length of the level sets in~\eqref{eq:Ar}. In the proof of the below proposition and in the following, we use the notation $\cB_\rho(x)\subset\dR^2$ for the open disk of radius $\rho > 0$ centered at $x\in\dR^2$.   
\begin{prop}\label{prop:Ar}
	Let the assumptions be as in Proposition~\ref{prop:convex} and the level sets $\cA_r$, $r\in(0,1)$, be defined as in~\eqref{eq:Ar}. Then there exist some constants $c_2 > c_1 > 0$, which  depend on $\Omg$, such that
	\[
		c_1 r \le |\cA_r| \le c_2 r,\qquad\text{for all}\,\, r\in (0,1),
	\] 
	where $|\cA_r|$ stands for the one-dimensional Hausdorff measure of $\cA_r$.
\end{prop}
\begin{proof}
	Recall that $x_\star\in\Omg$ is the point at which the torsion function $v_\Omg$ acquires its maximal value. As a direct consequence of the negative definiteness of the Hessian of $v_\Omg$ at the point $x_\star$ stated in Proposition~\ref{prop:convex}\,(ii),
	we get that
	for a sufficiently small $\eps > 0$ there exist constants $a_2 > a_1 > 0$ such that
	\[	
		a_1|x-x_\star|^2\le \max v_\Omg - v_\Omg(x) \le a_2|x-x_\star|^2,\qquad\text{for all}\,\, x\in\cB_\eps(x_\star).
	\]
	Using that $v_\Omg\in C^\infty(\ov\Omg)$ and that $v_\Omg$ has a unique global maximum at $x_\star$ we conclude that one can find $b_2 \in [a_2,\infty)$ and 
	$b_1 \in (0,a_1]$ such that
	\begin{equation}\label{eq:bndb}	
		b_1|x-x_\star|^2\le \max v_\Omg - v_\Omg(x) \le 	b_2|x-x_\star|^2,\qquad\text{for all}\,\, x\in\Omg.
	\end{equation}
	We also remark that by Proposition~\ref{prop:convex}\,(iv) one can choose $b_2 = \frac12$. However, the exact suitable value of $b_2$ is not essential for the argument.

	By the construction, the level set $\cA_r$ is the boundary of $\Omg_r$ in~\eqref{eq:Omgr}. In particular, by Proposition~\ref{prop:convex}\,(i), $\Omg_r$ is convex.
	
	Using the lower bound on $\max v_\Omg - v_\Omg(x)$ in~\eqref{eq:bndb}   
	we get the following inclusion
	\begin{equation}\label{eq:inclusion1}
		\Omg_r \subset \{x\in\Omg\colon 
		b_1|x-x_\star|^2 < r^2\max v_\Omg\} \subset\cB_{r\sqrt{\max v_\Omg/b_1}}(x_\star).
	\end{equation}
	Analogously, using the upper bound on $\max v_\Omg - v_\Omg(x)$ in~\eqref{eq:bndb}   
	we get the inclusion
	\begin{equation}\label{eq:inclusion2}
		\Omg_r \supset \{x\in\Omg\colon 
		b_2|x-x_\star|^2 < r^2\max v_\Omg\}
		= \cB_{r\sqrt{\max v_\Omg/b_2}}(x_\star),\qquad\text{for all}\,\,r\in(0,r_\star),
	\end{equation}
	where $r_\star\in(0,1)$ is chosen to be sufficiently small to ensure that $\cB_{r\sqrt{\max v_\Omg/b_2}}(x_\star)\subset\Omg$ for all $r\in(0,r_\star)$.
	
	Combining the monotonicity of the perimeter with respect to inclusion of convex sets~\cite[Lemma 2.2.2]{BB} with the standard formula for the perimeter of the disk we conclude from~\eqref{eq:inclusion1} and~\eqref{eq:inclusion2} that there exist constants $\wt{c}_2 > \wt{c}_1 > 0 $ such that for all $r\in(0,r_\star)$,
	\begin{equation}\label{eq:Ar_prelim}
		\wt{c}_1 r \le |\cA_r|\le \wt{c}_2 r.
	\end{equation}
	Note that again by~\cite[Lemma 2.2]{BB},
	the function $(0,1)\ni r\mapsto |\cA_r|$ is non-decreasing
	and bounded by $|\p\Omg|$ from above. Hence,
	the two-sided bound in~\eqref{eq:Ar_prelim} can be extended to all $r\in(0,1)$ with possibly different constants. Namely,
	 one can find $c_1\in (0,\wt{c}_1]$ and $c_2\in [\wt{c}_2,\infty)$ such that
	\[
	c_1 r \le |\cA_r|\le c_2 r,\qquad \text{for all}\,\, r\in (0,1),
	\]
	which is the desired inequality.
\end{proof}

\subsection{Magnetic Neumann Laplacian on a general domain}
In this subsection, we will introduce the magnetic Neumann Laplacian on a bounded, smooth, simply-connected domain
with the homogeneous magnetic field. In our analysis, we use the gauge of the magnetic filed based on the magnetic scalar potential. With this choice of the gauge, the boundary condition does not depend on the intensity of the magnetic field.

 Let us fix the intensity of the homogeneous magnetic field $b > 0$. We consider the unique {\it magnetic scalar potential} $\phi\in C^\infty(\ov\Omg)$ satisfying 
\begin{equation}\label{eq:scalar_potential}
	\begin{cases} 
		\Delta \phi = b,&\,\text{in}\,\,\Omg,\\
		\phi = 0,&\,\text{on}\,\,\p\Omg.
	\end{cases}
\end{equation}
In particular, we have the relation $\phi = -bv_\Omg$, where $v_\Omg\in C^\infty(\ov\Omg)$ is the torsion function of $\Omg$ introduced and discussed in Subsection~\ref{ssec:torsion}.

In terms of the magnetic scalar potential, the vector potential
\begin{equation}\label{eq:vector_potential}
\bA(x) := (-\p_2\phi, \p_1\phi)^\top
\end{equation}
is the unique solution in the Sobolev space $H^1(\Omg;\dC^2)$ of the boundary value problem:
\begin{equation}
\begin{cases}
	{\rm curl}\, \bA = b,&\,\text{in}\,\,\Omg,\\
	\div \bA = 0,&\,\text{in}\,\,\Omg,\\
	\nu \cdot \bA = 0,&\,\text{on}\,\,\p\Omg,
\end{cases}
\end{equation}	
see \cite[Section D.2.1]{FH} for details.
Following the lines of~\cite[\S 1.2]{FH}, we introduce the self-adjoint magnetic Neumann Laplacian $\sfH_b^\Omg$ in $L^2(\Omg)$ with the homogeneous magnetic field of intensity $b$ via the closed, non-negative, densely defined quadratic form
\begin{equation}\label{eq:form}
	\frh_b^\Omg[u]  := \int_{\Omg}|(-\ii\nb - \bA)u|^2\dd x,\qquad \dom\frh_b^\Omg := H^1(\Omg).
\end{equation}
Due to our special choice of the gauge for the vector potential of the magnetic field the operator $\sfH_b^\Omg$
is characterized by 
\begin{equation}\label{key}
\sfH_b^\Omg u = - (\nabla - \ii\bA)^2 u,\qquad \dom\sfH_b^\Omg = \big\{u\in H^2(\Omg)\colon \nu\cdot\nb u = 0\,\,\text{on}\,\,\p\Omg\big\}.	
\end{equation}
The spectrum of the non-negative operator $\sfH_b^\Omg$ is purely discrete and we denote by $\mu_1^b(\Omg) > 0$ its lowest eigenvalue. The fact that $\mu_1^b(\Omg)$ is strictly positive is verified \eg~in~\cite[Appendix C]{KL22}. According to the min-max principle~\cite[Theorem 1.28]{FLW23} we have the variational characterization
\begin{equation}\label{eq:variational}
	\mu_1^b(\Omg) = \inf_{u\in H^1(\Omg)\sm\{0\}} \frac{\displaystyle\int_{\Omg}|(-\ii\nb - \bA)u|^2\dd x}{\displaystyle\int_{\Omg}|u|^2\dd x}.
\end{equation}
As we already mentioned, since $\Omg$ is simply connected, $\mu_1^b(\Omg)$ is invariant under the change of the gauge of the vector potential for the homogeneous magnetic field of intensity $b >0$. Thus, by choosing the symmetric gauge as in the introduction we would get the same lowest eigenvalue for the magnetic Neumann Laplacian.

\subsection{Magnetic Neumann Laplacian on a disk}
The aim of this subsection is to collect properties of the
magnetic Neumann Laplacian on a disk. These properties will be extensively used in the proof of our main result.

Let $\cB = \cB_R\subset\dR^2$ be the disk of radius $R >0$
centered at the origin. It is not hard to verify that the solution of the system~\eqref{eq:scalar_potential} for the disk is given by
\begin{equation}
	\phi(x) = b\left(\frac{x^2_1+x_2^2}{4} - \frac{R^2}{4}\right),
\end{equation} 
where we used the convention $x = (x_1,x_2)$.
Thus, we find that the associated vector potential of the magnetic field reads as
\begin{equation}\label{eq:Adisk}
	\bA(x) = \frac{b}{2}(-x_2,x_1)^\top,
\end{equation}
and coincides with the symmetric gauge chosen in the introduction.

The magnetic Neumann Laplacian $\sfH_b^\cB$ on the disk $\cB$
acts in $L^2(\cB$) and is introduced via the quadratic form~\eqref{eq:form} with $\Omg = \cB$ and $\bA$ as in~\eqref{eq:Adisk}.

It is convenient to use polar coordinates $(r,\tt)$ on the disk $\cB$.
Following the lines of~\cite[Section 2]{KL22}, we introduce 
a complete family of mutually orthogonal projections
in the Hilbert space $L^2(\cB)$ defined in polar coordinates
by
\[
	(\Pi_m u)(r,\tt) := \frac{1}{2\pi}e^{\ii m\tt}\int_0^{2\pi} u(r,\tt') e^{-\ii m\tt'}\dd\tt',\qquad m\in\dZ.
\]
The space $\ran\Pi_m$ can be naturally identified with $L^2((0,R);r\der r)$ leading thus to the orthogonal decomposition
\begin{equation}\label{eq:orth}
	L^2(\cB)\simeq\bigoplus_{m\in\dZ} L^2((0,R);r\der r).
\end{equation}
Let us introduce for any $m\in\dZ$ the closed, densely defined and non-negative quadratic form in the Hilbert space $L^2((0,R);r\der r)$ by
\begin{equation}\label{eq:form_fiber}
\begin{aligned}
		\frh_{b,m}^R[f] & :=\int_0^R\bigg[|f'(r)|^2 +\frac{1}{r^2}\left(m-\frac{br^2}{2}\right)^2|f(r)|^2\bigg]r\dd r,\\
		\dom\frh_{b,m}^R & := \big\{f\colon f,f', mr^{-1}f\in L^2((0,R);r\der r)\big\}.
\end{aligned}
\end{equation}
The self-adjoint fiber operator in the Hilbert space $L^2((0,R);r\der r)$ associated with the above quadratic form
via the first representation theorem is denoted by $\sfH_{b,m}^R$, $m\in\dZ$. According to~\cite[Eq. (3.5)]{KL22} we have the orthogonal decomposition of the magnetic Neumann Laplacian $\sfH^\cB_b$, 
\begin{equation}\label{eq:orth_operator}
	\sfH^\cB_b \simeq\bigoplus_{m\in\dZ} \sfH_{b,m}^R
\end{equation}
with respect to~\eqref{eq:orth}. The spectra of all the fiber operators are purely discrete and we denote by $\mu_1^{b,m}(R) > 0$ the lowest eigenvalue of the fiber operator $\sfH_{b,m}^R$. It follows immediately from the orthogonal decomposition~\eqref{eq:orth_operator} that
\[
	\mu_1^b(\cB) = \inf_{m\in\dZ} \mu_1^{b,m}(R).
\]
In the next proposition we provide some basic properties of the lowest eigenvalue and the corresponding eigenfunction for the magnetic Neumann Laplacian on a disk with a weak homogeneous magnetic field.
\begin{prop}\label{prop:disk}
	Let $\sfH_b^\cB$ be the magnetic Neumann Laplacian on the disk $\cB=\cB_R\subset\dR^2$ of radius $R >0$ introduced via the quadratic form~\eqref{eq:form}.  Then the following hold.
\begin{myenum}
\item Under the assumption $bR^2 < 1$, it holds that 
$\mu_1^b(\cB) = \mu_1^{b,0}(R)$ and, in particular, a radial real-valued eigenfunction 
corresponds to the lowest eigenvalue of $\sfH_b^\cB$. %
\item $(0,b^{-1/2})\ni R\mapsto \mu_1^b(\cB_R)$ is a strictly increasing function.
\item There exists $b_\circ = b_\circ(R)>0$ such that the function
\[(0,b_\circ)\ni b\mapsto \frac{\mu_1^b(\cB_R)}{b^2}\]
		is strictly decreasing.
	\end{myenum}
\end{prop}
\begin{proof}
(i) The statement of this item is proved in~\cite[Proposition 3.3]{KL22}. 

\smallskip

\noindent (ii) 
According to item (i), under the assumption $bR^2 < 1$ the lowest eigenvalue of the magnetic Laplacian on $\cB$ corresponds to the fiber labelled by $m = 0$. Hence, we get using the min-max principle that
\begin{equation}\label{eq:variational_radial}
	\mu_1^b(\cB_R) = \inf_{f\in \dom\frh_{b,0}^R\sm\{0\}}
	\frac{\displaystyle\int_0^R\Big(|f'(r)|^2 +\frac{b^2r^2}{4}|f(r)|^2\Big)r\dd r}{\displaystyle\int_0^R|f(r)|^2 r\dd r},\qquad (bR^2 < 1).
\end{equation}
Using the characteristic function of the interval $[0,R]$ as a trial function we get the following upper bound on the lowest eigenvalue
\begin{equation}\label{eq:upper_bound_disk}
	\mu_1^b(\cB_R) \le \frac{\displaystyle\int_0^R \frac{b^2r^3}{4}\dd r}{\displaystyle\int_0^R r\dd r} = \frac{b^2R^2}{8}.
\end{equation}
Let us choose arbitrary $R_2 \in (0,b^{-1/2})$ and $R_1 \in [\frac{R_2}{\sqrt{2}},R_2)$.
We introduce an auxiliary closed, densely defined, non-negative quadratic form in the Hilbert space $L^2((0,R_2);r\dd r)$
\[
	\begin{aligned}
		\frt_{b}^{R_1,R_2}[f] & :=
		 \int_0^{R_1}\bigg[
		|f'_1(r)|^2 +\frac{b^2r^2}{4}|f_1(r)|^2\bigg]r\dd r+\int_{R_1}^{R_2}\bigg[
		|f'_2(r)|^2 +\frac{b^2r^2}{4}|f_2(r)|^2\bigg]r\dd r,\\
		\dom\frt_b^{R_1,R_2} &:= \Big\{f=f_1\oplus f_2\colon
		f_1,f_1'\in L^2((0,R_1);r\der r), f_2,f_2'\in L^2((R_1,R_2);r\der r)\Big\}.
	\end{aligned}
\]
It is straightforward to check that $\frt^{R_1,R_2}_b\prec \frh_{b,0}^{R_2}$ in the sense of ordering of quadratic forms thanks to the inclusion $\dom\frh_{b,0}^{R_2}\subset\dom\frt^{R_1,R_2}_b$ and the identity $\frh_{b,0}^{R_2}[f] = \frt^{R_1,R_2}_b[f]$ valid for any $f\in\dom\frh_{b,0}^{R_2}$. Let $\sfT^{R_1,R_2}_b$ be the self-adjoint operator associated with the quadratic form $\frt_b^{R_1,R_2}$ via the first representation theorem. By the min-max principle we infer that the lowest eigenvalue $\mu_1(\sfT^{R_1,R_2}_b)$ of $\sfT^{R_1,R_2}_b$ satisfies
\begin{equation}\label{eq:bound_NB}
	\mu_1(\sfT^{R_1,R_2}_b)\le \mu_1^b(\cB_{R_2}).
\end{equation}
The operator $\sfT^{R_1,R_2}_b$ can be decomposed into the orthogonal sum
\begin{equation}\label{eq:orthogonal}
	\sfT^{R_1,R_2}_b = \sfH_{b,0}^{R_1}\oplus \wh\sfT^{R_1,R_2}_b,
\end{equation}
where $\sfH_{b,0}^{R_1}$ is the self-adjoint fiber operator in $L^2((0,R_1);r\der r)$ associated with the quadratic form~\eqref{eq:form_fiber} with $R = R_1$ and $m=0$ while
the self-adjoint operator $\wh\sfT^{R_1,R_2}_b$ in $L^2((R_1,R_2);r\der r)$ is associated with the quadratic form
\[
	\{f\colon f,f'\in L^2((R_1,R_2);r\der r)\}\ni f\mapsto \int_{R_1}^{R_2}\bigg[|f'(r)|^2+\frac{b^2r^2}{4}|f(r)|^2\bigg] r\dd r.
\]
Using that $R_1 \ge \frac{R_2}{\sqrt{2}}$ we end up with the bound
\begin{equation}\label{eq:lower_bound_annulus}
	\mu_1(\wh\sfT^{R_1,R_2}_b) > \frac{b^2R_2^2}{8}, 
\end{equation}
where $\mu_1(\wh\sfT^{R_1,R_2}_b)$ denotes the lowest eigenvalue of $\wh\sfT^{R_1,R_2}_b$. 

As a by-product of the above analysis we also get that the case of equality in~\eqref{eq:bound_NB} can not occur. Indeed, if the equality $\mu_1(\sfT^{R_1,R_2}_b) = \mu_1^b(\cB_{R_2})$ would hold, then by~\cite[Chapter~10, Theorem 1]{BS87} an eigenfunction $h\in\dom\frh_{b,0}^{R_2}$ of $\sfH_{b,0}^{R_2}$ corresponding to its lowest eigenvalue would simultaneously be the ground state of $\sfT^{R_1,R_2}_b$. In view of  \eqref{eq:upper_bound_disk} and~\eqref{eq:lower_bound_annulus} we would conclude that $h|_{[R_1,R_2]} \equiv 0$. Taking into account that by~\cite[Remark 3.1]{KL22}
the function
$h$ satisfies the ordinary differential equation
\[
-h''(r) - \frac{h'(r)}{r} +\frac{b^2r^2}{4}h(r) = \mu_1^b(\cB_{R_2})h(r),\qquad \text{for}\,\, r\in(0,R_2),
\]
we would get by the uniqueness of the solution of the 
initial value problem for the above ordinary differential equation that $h\equiv 0$, leading to a contradiction. 

Combining the upper bound~\eqref{eq:upper_bound_disk} with the lower bound~\eqref{eq:lower_bound_annulus} and the eigenvalue inequality \eqref{eq:bound_NB} we get taking into account the orthogonal decomposition~\eqref{eq:orthogonal} that
\begin{equation}\label{eq:bnd}
	\mu_1^b(\cB_{R_1}) = \mu_1^{b,0}(R_1) = \mu_1(\sfH_{b,0}^{R_1}) < \mu_1^b(\cB_{R_2}).
\end{equation}
The latter bound implies that $(0,b^{-1/2})\ni R\mapsto \mu_1^b(\cB_R)$ is a strictly increasing function.

\smallskip

\noindent (iii)
The definition of the operator $\sfH_b^\cB$
	and of its lowest eigenvalue $\mu_1^b(\cB)$ can be naturally extended to negative $b$. In particular, the relation $\mu_1^{-b}(\cB) = \mu_1^b(\cB)$ holds for any $b > 0$.
Note that by~\cite[Section 1.5]{FH} the function $\dR\ni b\mapsto\mu_1^b(\cB)$ is real analytic in a neighbourhood of the origin.  Moreover, there exists  $b_\star = b_\star(R) > 0$ and  the real coefficients $\{c_j\}_{j=2}^\infty$ (which depend on $R$) so that
\begin{equation}\label{eq:expansionbeta}
	\mu_1^b(\cB_R) = \sum_{j=2}^\infty c_j b^j,\qquad \text{for all}\,\, b\in (0,b_\star),
\end{equation}
where $c_2 = \frac{R^2}{8}$. It follows from the min-max principle with the characteristic function of $\cB_R$ as the trial function that for all $b > 0$
\[
\mu_1^b(\cB_R) < \frac{b^2 R^2}{8}.
\]
The inequality is strict because the characteristic function of $\cB_R$ is not the ground state of $\sfH_b^\cB$ for all $b > 0$.
Hence,  we infer that the first non-zero coefficient $c_j$
with index $j > 2$ in the expansion~\eqref{eq:expansionbeta} is negative.   Let $j_0 > 2$ be the index of this coefficient.  Thus, there exists
a sufficiently small $b_\circ = b_\circ(R) \in (0,b_\star)$ such that the function
\begin{equation}\label{eq:ratio}
	(0,b_\circ)\ni b\mapsto u(b):= \frac{\mu_1^b(\cB_R)}{b^2}
\end{equation}
is strictly decreasing.   Indeed, the  function defined by $v(b):=\sum_{j\geq j_0}c_{j}b^{j-j_0+1}$ 
is real analytic in a neighbourhood of the origin and satisfies $v'(0)=c_{j_0}<0$,  so there is $b_\circ>0$ such that $v$ is negative-valued and decreasing on $(0,b_\circ)$,  and for $0<b_1<b_2<b_\circ$,  we have 
\[ \frac{u(b_1)-c_2}{u(b_2)-c_2}=\left(\frac{b_1}{b_2}\right)^{j_0-3}\frac{v(b_1)}{v(b_2)}<1,\]
from which the inequality $u(b_2) < u(b_1)$ follows in view of $u(b_2) -c_2 < 0$.
\end{proof}
\begin{remark}
	For $b = 1$, we get from Proposition~\ref{prop:disk} that
	a radial eigenfunction corresponds to the lowest eigenvalue of $\sfH^\cB_b$ for all $R\in(0,1)$ and that $R\mapsto \mu_1^b(\cB_R)$ is strictly increasing on the interval $(0,1)$.
	The numerical computations in~\cite[Figure 3]{CLPS23} show that both properties hold up to $R\approx 2$.
\end{remark}

\section{A bound on the lowest magnetic Neumann eigenvalue  in terms of the torsion function}
\label{sec:bound}
In this section we obtain an upper bound on the lowest magnetic Neumann eigenvalue on a bounded, convex domain with moderate intensity of the magnetic field. The bound is given in terms of the torsion function and the lowest 
magnetic Neumann eigenvalue of the disk with the same
maximal value of the torsion function as the domain.  The employed trial function is constant on the level curves of the torsion function.

Throughout this section we will rely on a variant of the co-area formula.
Recall that the co-area formula applied in two dimensions (see~\cite[\S 2.12]{AFP} and also~\cite[Theorem 3.1]{F59}) to an open set $\cU\subset\dR^2$, a Lipschitz continuous real-valued function $f\colon\cU\arr\dR$, and a measurable function $g\colon\cU\arr[0,\infty)$ gives
\begin{equation}\label{eq:coarea}
	\int_\cU g(x)|\nabla f(x)|\dd x =
	\int_\dR\int_{f^{-1}(t)} g(x)\dd \cH^1(x)\dd t,
\end{equation}
where $\cH^1$ in the inner integral on the right-hand side is the one-dimensional Hausdorff  measure on the level set $f^{-1}(t):=\{x\in\cU\colon f(x) = t\}$. We remark that by 
a consequence of the Sard theorem (see~\cite[Theorem 2.5\,(iv)]{ABC13}) for almost all $t\in\dR$
the connected components of $f^{-1}(t)$ are either points or simple closed Lipschitz curves. 

In the following, $\Omg\subset\dR^2$ is a bounded, convex, $C^\infty$-smooth domain and $v_\Omg\in C^\infty(\ov\Omg)$ is the torsion function of $\Omg$ introduced as in Subsection~\ref{ssec:torsion}. 
Two characteristics of the domain $\Omg$ play an important role in our analysis
\begin{equation}\label{eq:AB}
	F(\Omg) := \esinf_{r\in(0,1)}\int_{\cA_r} \frac{1}{|\nb v_\Omg(x)|}\dd\cH^1(x),\qquad
	G(\Omg) := \essup_{r\in(0,1)} \int_{\cA_r}\frac{|\nb v_\Omg(x)|}{r^2}\dd\cH^1(x),
\end{equation}
where the level set $\cA_r$ is defined in~\eqref{eq:Ar}.
\begin{lem}\label{lem:FG}
	Let $\Omg\subset\dR^2$ be a bounded, convex, $C^\infty$-smooth domain. Then the quantities $F(\Omg)$ and $G(\Omg)$ introduced in~\eqref{eq:AB} are finite and non-zero.
\end{lem}
\begin{proof}
	It suffices to show that $F(\Omg) > 0$ and that $G(\Omg) < \infty$. Using the estimate on the gradient of the torsion function in Proposition~\ref{prop:convex}\,(iii)
	and the lower bound on $|\cA_r|$ in Proposition~\ref{prop:Ar} we get
	\[
		F(\Omg) \ge \frac{1}{\sqrt{2\max v_\Omg}}
			\essinf_{r\in(0,1)}\frac{|\cA_r|}{r} \ge \frac{c_1}{\sqrt{2\max v_\Omg}} > 0.  
	\]
	Analogously, we get combining the bound in Proposition~\ref{prop:convex}\,(iii) now with the upper bound on $|\cA_r|$ in Proposition~\ref{prop:Ar} that
	\[
		G(\Omg) \le \sqrt{2\max v_\Omg}\, \esssup_{r\in(0,1)}\frac{|\cA_r|}{r}\le 
		c_2\sqrt{2\max v_\Omg} < \infty. \qedhere
	\]
\end{proof}
Now we are ready to formulate and prove the main result of this section and of the whole paper.
\begin{thm}\label{thm:main}
	Let $\Omg\subset\dR^2$ be a bounded, convex, $C^\infty$-smooth domain. Let the torsion function $v_\Omg\in C^\infty(\ov\Omg)$ of $\Omg$ be introduced as
	the solution of~\eqref{eq:torsion} and the associated quantities $F(\Omg)$ and $G(\Omg)$ be as in~\eqref{eq:AB}. Under the assumption $b|\Omg| < \pi$, the lowest eigenvalue $\mu_1^b(\Omg)$ of the magnetic Neumann Laplacian on $\Omg$ can be estimated as
	\[
		\mu_1^b(\Omg) \le \frac{1}{\max v_\Omg}\frac{G(\Omg)}{F(\Omg)}\mu_1^{b}(\cB_{2\sqrt{\max v_\Omg}}).
	\]
\end{thm}

\begin{remark}
The constant $C(\Omg)$ and the radius $\rho(\Omg)$ in~\eqref{eq:equivalent_bound} are thus expressed as
\[
	C(\Omg) = \frac{1}{\max v_\Omg}\frac{G(\Omg)}{F(\Omg)}\qquad\text{\and}\qquad\rho(\Omg) = 2\sqrt{\max v_\Omg},
\]
and the inequality in Theorem~\ref{thm:main} can be written as
\[
	\mu_1^b(\Omg)\le C(\Omg)\mu_1^b(\cB_{\rho(\Omg)}).
\]
We remark that the maximum of the torsion function for the disk $\cB_{\rho(\Omg)}$ is equal to $\max v_\Omg$.
It follows from~\eqref{eq:max_torsion}
that $\rho(\Omg)\le R$ where $R$ is the radius of the disk
$\cB\subset\dR^2$ with the same area as $\Omg$. Using the monotonicity property in Proposition~\ref{prop:disk}\,(ii)
we arrive at the following consequence
\[
	\mu_1^b(\Omg) \le C(\Omg)\mu_1^b(\cB).
\]
In Section~\ref{sec:geom-cst} we will show that $C(\Omg) \ge 1$.
\end{remark}
\begin{proof}[Proof of Theorem~\ref{thm:main}]
	Let us consider an auxiliary function $\psi\colon\Omg\arr[0,1]$ defined by
	\begin{equation}\label{eq:psi}
	\psi(x) := \sqrt{1-\frac{v_\Omg(x)}{\max v_\Omg}}.
	\end{equation} 
	The function $\psi$ attains the unit value on the boundary of $\Omg$ and vanishes at the point, where the torsion function $v_\Omg$ is maximal.
	Clearly, the function $\psi$ is continuous up to the boundary of $\Omg$ thanks to continuity of the torsion function. Moreover, for the gradient of $\psi$ we get
	 the following expression
	\begin{equation}\label{eq:nabla_psi}
		\nb\psi(x) = -\frac{\nb v_\Omg(x)}{2\max v_\Omg}\left(1-\frac{v_\Omg(x)}{\max v_\Omg}\right)^{-1/2} = -\frac{\nb v_\Omg(x)}{2(\max v_\Omg)\psi(x)}.
	\end{equation}
	Using the bound in Proposition~\ref{prop:convex}\,(iii)
	we get
	\[
		|\nb \psi(x)| = \frac{|\nb v_\Omg(x)|}{2\sqrt{\max v_\Omg}\sqrt{\max v_\Omg - v_\Omg(x)}} \le \frac{1}{\sqrt{2\max v_\Omg}},\qquad \text{for all}\,\, x\in\Omg\sm\{x_\star\}.
	\]
	In particular, we conclude that the function $\psi$ is Lipschitz.
	As a trial function for the magnetic Neumann Laplacian on $\Omg$ we select
	\[
		u_f(x) := f\big(\psi(x)\big),\qquad x\in\Omg,
	\]
	where $f\in C^\infty([0,1])$ is a generic and real-valued function.
	Thanks to regularity of $f$ and Lipschitz continuity of $\psi$ we obtain that $u_f\in H^1(\Omg)$. We remark that the function $u_f$ is also real-valued. 
	By the variational characterization~\eqref{eq:variational} we get
	\begin{equation}\label{eq:bnd2}
		\mu^b_1(\Omg) \le \inf_{\begin{smallmatrix}f\in C^\infty([0,1];\dR)\\ f\ne0\end{smallmatrix}}\frac{\displaystyle
			\int_\Omg|\nb u_f(x)|^2\dd x + b^2\int_\Omg|\nb v_\Omg(x)|^2|u_f(x)|^2\dd x}{\displaystyle\int_\Omg|u_f(x)|^2\dd x},
	\end{equation}
	where we used that $u_f$ is real-valued and applied the formula~\eqref{eq:vector_potential} for the vector potential of the magnetic field.
	Next, we estimate the three integrals entering the upper bound~\eqref{eq:bnd2}.
	For the integral in the denominator we obtain combining the co-area formula and the identity~\eqref{eq:nabla_psi} 
	\begin{equation}\label{eq:integral1}
		\begin{aligned}
			\int_\Omg|u_f(x)|^2\dd x & =
			\int_\Omg\frac{|f(\psi(x))|^2}{|\nb\psi(x)|}|\nb\psi(x)|\dd x \\
			&=\int_0^1|f(r)|^2\int_{\cA_r}\frac{1}{|\nb \psi(x)|}\dd\cH^1(x)\dd r \\
			&=
			2\max v_\Omg
			\int_0^1|f(r)|^2r\int_{\cA_r}\frac{1}{|\nb v_\Omg(x)|}\dd\cH^1(x)\dd r\\
			&\ge 
			2F(\Omg)\max v_\Omg\int_0^1  |f(r)|^2r\dd r,
		\end{aligned}
	\end{equation}
	where we used in between that $\cA_r = \{x\in\Omg\colon\psi(x) =r\}$.
	For the first integral in the numerator in~\eqref{eq:bnd2} we get again applying the co-area formula
	\begin{equation}\label{eq:integral2}
	\begin{aligned}
		\int_\Omg|\nb u_f(x)|^2\dd x &=\int_\Omg |f'(\psi(x))|^2|\nb\psi(x)|^2\dd x \\
		&= \int_0^1 |f'(r)|^2 \int_{\cA_r}|\nb \psi(x)|\dd\cH^1(x)\dd r\\
		&= \frac{1}{2\max v_\Omg} \int_0^1|f'(r)|^2r \int_{\cA_r}\frac{|\nb v_\Omg(x)|}{r^2}\dd\cH^1(x)\dd r\\
		& \le \frac{G(\Omg)}{2\max v_\Omg}\int_0^1|f'(r)|^2 r\dd r.
	\end{aligned}	
	\end{equation}
	Finally, for the second integral in the numerator in the right hand side of~\eqref{eq:bnd2} we get
	\begin{equation}\label{eq:integral3}
		\begin{aligned}
			\int_\Omg|\nb v_\Omg(x)|^2|u_f(x)|^2\dd x & =
			4(\max v_\Omg)^2\int_\Omg |f(\psi(x))|^2|\nb \psi(x)|^2|\psi(x)|^2\dd x\\
			&= 2\max v_\Omg \int_0^1 |f(r)|^2 r^3\int_{\cA_r}\frac{|\nb v_\Omg(x)|}{r^2}\dd\cH^1(x)\dd r\\
			&\le 2G(\Omg)\max v_\Omg \int_0^1|f(r)|^2 r^3\dd r.  
		\end{aligned}
	\end{equation}
	Plugging the estimates~\eqref{eq:integral1},~\eqref{eq:integral2}, and~\eqref{eq:integral3} into the upper bound~\eqref{eq:bnd2} we arrive at
	\begin{equation*}\label{key}
		\mu_1^b(\Omg) \le \frac{1}{4(\max v_\Omg)^2}\frac{G(\Omg)}{F(\Omg)}
		\inf_{\begin{smallmatrix}
				f\in C^\infty([0,1];\dR)\\ f\ne 0\end{smallmatrix}} \frac{\displaystyle\int_0^1|f'(r)|^2r\dd r +4(\max v_\Omg)^2 b^2\int_0^1 |f(r)|^2r^3\dd r}{\displaystyle\int_0^1 |f(r)|^2r\dd r}.
	\end{equation*}
	Thanks to the condition $b|\Omg| < \pi$ and the inequality~\eqref{eq:max_torsion} we get that
	$4(\max v_\Omg)b < 1$. Hence, using that $C^\infty([0,1])$ is a core for the quadratic form $\frh_{b,0}^1$ defined in~\eqref{eq:form_fiber}, we obtain by~\eqref{eq:variational_radial} that
	\begin{equation}\label{eq:almost}
		\mu_1^b(\Omg) \le \frac{1}{4(\max v_\Omg)^2}\frac{G(\Omg)}{F(\Omg)}\mu_1^{4b\max v_\Omg}(\cB_1).
	\end{equation}
	Using the scaling property,
	$\mu_1^{b'}(\cB_1) = t^2\mu_1^{b'/t^2}(\cB_t)$ for all $t,b' > 0$ (see~\cite[page 1]{CLPS23}),
	we get
	\[
	\frac{1}{4\max v_\Omg}\mu_1^{4b\max v_\Omg}(\cB_1) = 
	\mu_1^b(\cB_{2\sqrt{\max v_\Omg}}).
	\]
	Hence, we derive from~\eqref{eq:almost} that
	\[
		\mu_1^b(\Omg) \le \frac{1}{\max v_\Omg}\frac{G(\Omg)}{F(\Omg)}\mu_1^b(\cB_{2\sqrt{\max v_\Omg}}),
	\]	
	which is the desired inequality.
\end{proof}

\section{Estimates for the geometric constants}\label{sec:geom-cst}
In this section,  we study the geometric constants \(F(\Omega)\) and \(G(\Omega)\) appearing in Theorem~\ref{thm:main}.   We will give a geometric lower bound on \(G(\Omega)\), a geometric upper bound on $F(\Omg)$ and a universal lower bound on \(F(\Omega)\).  

In the next lemma we obtain alternative formulas for
$F(\Omg)$ and $G(\Omg)$.
\begin{lem}\label{lem:FG2}
	Let $\Omg\subset\dR^2$ be a bounded, convex, $C^\infty$-smooth domain with the torsion function $v_\Omg\in C^\infty(\ov\Omg)$.
	Let $\Omg_r$, $r\in(0,1)$, be superlevel sets of $v_\Omg$ defined as in~\eqref{eq:Omgr}.
	Then the function $r\mapsto |\Omg_r|$ and the quantities $F(\Omg)$ and $G(\Omg)$ introduced in~\eqref{eq:AB} satisfy:
	\begin{myenum}
		\item $(0,1)\ni r\mapsto |\Omg_r|$ is $C^\infty$-smooth;
		\item $G(\Omg) = \sup_{r\in (0,1)}\frac{|\Omg_r|}{r^2}$;
		\item $F(\Omg) = \frac{1}{2\max v_\Omg}\inf_{r\in(0,1)}
		\frac{\frac{\dd}{\dd r}|\Omg_r|}{ r}$.
	\end{myenum}
\end{lem}
\begin{proof}
	(i)
	Recall that $x_\star\in\Omg$ denotes the point where the function $v_\Omg$ attains its unique global maximum.
	Since by Proposition~\ref{prop:convex}\,(i) there holds $\nb v_\Omg(x) \ne 0$ for all $x\in\Omg\sm\{x_\star\}$,
	we get that all the values of the function $v_\Omg$
	in the interval $(0,\max v_\Omg)$ are regular. Hence, we conclude from~\cite[Chapter IV, Theorem 1]{Ch} that the function $(0,1)\ni r\mapsto |\Omg_r|$ is $C^\infty$-smooth.
	
	\smallskip
	
	\noindent (ii) By the implicit function theorem~\cite[Theorem 3.3.1]{KP02}, 	the boundary $\cA_r$ of $\Omg_r$ is $C^\infty$-smooth.
	Let us denote by $\nu_r$ the outer unit normal for $\Omg_r$.  Hopf's lemma yields that $\frac{\p v_\Omg}{\p\nu_r}(x)<0$ on \(\cA_r\).
	Using the fact that $\cA_r$ is a level set of $v_\Omg$
	we get that the normal derivative $\frac{\p v_\Omg}{\p\nu_r}(x)$ of the torsion function at $x\in\cA_r$ equals $-|\nb v_\Omg(x)|$. 
	Then we can rewrite the formula for $G(\Omg)$ in~\eqref{eq:AB} as follows
	\[
	\begin{aligned}
		G(\Omg) &= \essup_{r\in(0,1)} \left(-\frac{1}{r^2}\int_{\cA_r}\frac{\p v_\Omg}{\p\nu_r}(x)\dd\cH^1(x)\right)\\
		&
		= \essup_{r\in(0,1)} \left(-\frac{1}{r^2}\int_{\Omg_r}\Delta v_\Omg(x)\dd x\right) \\
		& =\essup_{r\in(0,1)} \left(\frac{1}{r^2}\int_{\Omg_r}\dd x\right)
		=\essup_{r\in(0,1)} \frac{|\Omg_r|}{r^2},
	\end{aligned}
	\]  
	where we performed the integration by parts and used that $-\Delta v_\Omg = 1$. In view of the smoothness of the function $(0,1)\ni r\mapsto|\Omg_r|$ shown in item~(i) we are allowed
	to replace the essential supremum by the ordinary supremum. Thus, we have
	\[
		G(\Omg) = \sup_{r\in(0,1)} \frac{|\Omg_r|}{r^2}.
	\] 
	
	\smallskip
	
	\noindent (ii)
	By the co-area formula~\eqref{eq:coarea} with $\psi$ as in~\eqref{eq:psi} we get for any $r\in(0,1)$ that
	\[
	|\Omg_r| = \int_0^r\int_{\cA_t}\frac{1}{|\nb \psi(x)|}\dd\cH^1(x)\dd t = 2\max v_\Omg \int_0^r t\int_{\cA_t}\frac{1}{|\nb v_\Omg(x)|}\dd \cH^1(x)\dd t. 
	\]
	Hence, we obtain that the derivative of the function $(0,1)\ni r\mapsto |\Omg_r|$ is
	given by the formula
	\[
	\frac{\dd}{\dd r}|\Omg_r| = 2\max v_\Omg r\int_{\cA_r}\frac{1}{|\nb v_\Omg(x)|}\dd \cH^1(x).
	\]
	As a result we get that
	\[
	\int_{\cA_r}\frac{1}{|\nb v_\Omg(x)|}\dd \cH^1(x)=
	\frac{\frac{\dd}{\dd r}|\Omg_r|}{2\max v_\Omg r},
	\]
	and the claim follows from the definition of $F(\Omg)$ in~\eqref{eq:AB}, in which we are allowed to replace the essential infimum by the ordinary infimum taking smoothness of the function $r\mapsto |\Omg_r|$ into account.
\end{proof}

Let $D^2 v_\Omg$
denote the Hessian of the torsion function. It follows from the equation $-\Delta v_\Omg = 1$ that $\Tr(D^2 v_\Omg(x)) = -1$ for all $x\in\Omg$.
Recall that $x_\star\in\Omg$ is the point where torsion function $v_\Omg$ attains its maximal value. By Proposition~\ref{prop:convex}\,(ii), $\cH_\star := D^2 v_\Omg(x_\star)\in\dR^{2\times2}$ is negative definite.
In the next proposition we obtain a lower bound for $G(\Omg)$ in terms of the maximal value of the torsion function, the determinant of $\cH_\star$ and the area of the domain.  
\begin{prop}\label{prop:GOmg}
	Let $\Omg\subset\dR^2$ be a bounded, convex, $C^\infty$-smooth domain. Let  $v_\Omg\in C^\infty(\ov\Omg)$ be its torsion function.
	Let $G(\Omg)$ be defined as in~\eqref{eq:AB}.
	Let $\cH_\star\in\dR^{2\times 2}$ be the Hessian of the torsion function $v_\Omg$ at the point of its maximum.
	Then there holds 
	\[
		G(\Omg)\geq \max\left\{ \frac{2\pi \max v_\Omg}{\sqrt{\det\cH_\star}},|\Omega|\right\}.
	\]
\end{prop}
\begin{proof}
	Let us denote by $\xi(\Omg) \in (0,\frac12]$ the absolute value of the largest eigenvalue of $\cH_\star$.
	Taking into account that $\Tr\cH_\star = -1$ and that $\cH_\star$ is negative definite, we get that the absolute value of the second eigenvalue of $\cH_\star$ is given by $1-\xi(\Omg)$.
	
	We choose the coordinate system in $\dR^2$ so that  $x_\star = 0$ (the maximum of the torsion function is attained at the origin) and that the Hessian of the torsion function at the origin is a diagonal matrix
	with entries $-\xi(\Omg)$ and $-(1-\xi(\Omg))$.
	By Taylor's expansion we get
	\[
		v_\Omg(x_1,x_2) = \max v_\Omg - \frac{\xi(\Omg)}{2}x_1^2 -\frac{1-\xi(\Omg)}{2}x_2^2 + \cO(|x|^3),\qquad x = (x_1,x_2)\arr 0.
	\]
	Therefore, for any $\eps < \xi(\Omg)/2$ there exists a sufficiently small $r_\eps > 0$ such that 
	\[
		\left|v_\Omg(x_1,x_2) - \max v_\Omg
		+ \frac{\xi(\Omg)}{2}x_1^2 +\frac{1-\xi(\Omg)}{2}x_2^2\right|\le \eps|x|^2
	\]
	for all $x=(x_1,x_2)\in \cB_{r_\eps}(0)$. Hence, we get that for all sufficiently small $r > 0$ there holds
	\[
	\begin{aligned}
		\Omg_r&\supset\left\{(x_1,x_2)\in\dR^2\colon 
		\left(\frac{\xi(\Omg)}{2}+\eps\right)x_1^2+
		\left(\frac{1-\xi(\Omg)}{2}+\eps\right)x_2^2 < \max v_\Omg r^2\right\},\\		
		\Omg_r&\subset			
				\left\{(x_1,x_2)\in\dR^2\colon
		\left(\frac{\xi(\Omg)}{2}-\eps\right)x_1^2+
		\left(\frac{1-\xi(\Omg)}{2}-\eps\right)x_2^2 < \max v_\Omg r^2\right\}.		
	\end{aligned}
	\]			
	Hence, we get using the formula for the area of an ellipse
	\[
		\frac{\pi\max v_\Omg r^2}
		{\sqrt{\left(\frac{\xi(\Omg)}{2}+\eps\right)
				\left(\frac{1-\xi(\Omg)}{2}+\eps\right)}}
			\le |\Omg_r|\le 
					\frac{\pi\max v_\Omg r^2}
			{\sqrt{\left(\frac{\xi(\Omg)}{2}-\eps\right)
					\left(\frac{1-\xi(\Omg)}{2}-\eps\right)}}
	\]
	for all sufficiently small $r >0$. Since $\eps > 0$ can be chosen arbitrarily small we get using Lemma~\ref{lem:FG2}\,(ii) that
	\begin{equation}\label{eq:Glim}
		G(\Omg) \geq  \lim_{r\arr0^+}\frac{|\Omg_r|}{r^2} = 
		\frac{2\pi \max v_\Omg}{\sqrt{\xi(\Omg)(1-\xi(\Omg))}} = \frac{2\pi \max v_\Omg}{\sqrt{\det \cH_\star}}.
	\end{equation}
	Moreover,  from the formula for $G(\Omg)$ in Lemma~\ref{lem:FG2}\,(ii) combined with the fact
	\[
		\lim_{r\arr1^-}\frac{|\Omg_r|}{r^2} = |\Omg|,
	\]
	we get
	\[
		G(\Omega) \ge |\Omg|.\qedhere
	\]
\end{proof}
In the next proposition we obtain a universal lower bound on $F(\Omg)$ and an upper bound on $F(\Omg)$ in terms of the Hessian $\cH_\star$.
\begin{prop}\label{prop:FOmg}
	Let $\Omg\subset\dR^2$ be a bounded, convex, $C^\infty$-smooth domain. Let  $v_\Omg\in C^\infty(\ov\Omg)$ be its torsion function.
	Let $F(\Omg)$ be defined as in~\eqref{eq:AB}.
	Let $\cH_\star\in\dR^{2\times 2}$ be the Hessian of the torsion function $v_\Omg$ at the point of its maximum.
	Then it holds that
	\[
	4\pi \leq F(\Omg)\leq \frac{2\pi}{\sqrt{\det\cH_\star}}.
	\]
\end{prop}
\begin{proof}
	First, we will show the lower bound on $F(\Omg)$.
	Using the geometric isoperimetric inequality, the Cauchy-Schwarz inequality, and the computation based on the integration by parts as in the proof of Lemma~\ref{lem:FG2}\,(ii) we get
	\begin{equation}\label{eq:isop_ratio}
		\begin{aligned}
			\int_{\cA_r}\frac{1}{|\nb v_\Omg(x)|}\dd\cH^1(x)&=
			\frac{\displaystyle\left(\int_{\cA_r}\frac{1}{|\nb v_\Omg(x)|}\dd\cH^1(x)\right)\left(\int_{\cA_r}|\nb v_\Omg(x)|\dd\cH^1(x)\right)}{
				\displaystyle\int_{\cA_r}|\nb v_\Omg(x)|\dd\cH^1(x)}\\
			&\ge \frac{|\cA_r|^2}{-
				\displaystyle\int_{\cA_r}\frac{\p v_\Omg}{\p\nu_r}(x)\dd\cH^1(x)} = \frac{|\cA_r|^2}{|\Omg_r|} \ge 4\pi.
		\end{aligned}
	\end{equation}
	Thus, it follows from the formula for $F(\Omg)$ in~\eqref{eq:AB} that $F(\Omg) \ge 4\pi$.
	
	Next we will show the upper bound on $F(\Omg)$. Using 
	the formula for $F(\Omg)$ in Lemma~\ref{lem:FG2}\,(iii), L'Hospital's rule~\cite[Theorem II]{T52} and the limit in~\eqref{eq:Glim} we arrive at
	\[
		F(\Omg) \le \frac{1}{\max v_\Omg}\liminf_{r\arr0^+}\frac{\frac{\dd}{\dd r}|\Omg_r|}{2r} \le \frac{1}{\max v_\Omg}\lim_{r\arr0^+}\frac{|\Omg_r|}{r^2} = \frac{2\pi}{\sqrt{\det\cH_\star}}.
	\]
	We remark that in the second step of the above computation we have inequality, because we have not established existence of the limit 
	$\lim_{r\arr 0^+}\frac{\frac{\dd}{\dd r}|\Omg_r|}{r}$.
\end{proof}
From the bounds in Propositions~\ref{prop:GOmg} and~\ref{prop:FOmg} we get the following consequence.
\begin{cor}
	Let $\Omg\subset\dR^2$ be a bounded, convex, $C^\infty$-smooth domain with the torsion function $v_\Omg$.
	Let $F(\Omg)$ and $G(\Omg)$ be defined as in~\eqref{eq:AB}. Then the factor $C(\Omg) = \frac{1}{\max v_\Omg}\frac{G(\Omg)}{F(\Omg)}$ in Theorem~\ref{thm:main} satisfies $C(\Omg)\ge 1$.	
\end{cor}
\begin{proof}
	The inequality $C(\Omg)\ge1$ follows via a combination of the lower bound on $G(\Omg)$ in Proposition~\ref{prop:GOmg} and of the upper bound on $F(\Omg)$ in Proposition~\ref{prop:FOmg}.
\end{proof}
\begin{exam}\label{ex:square}
We  discuss a special class of convex domains $\Omg\subset\dR^2$ satisfying the following symmetries
\begin{equation}\label{eq:symm}
	\Omg = \{(-x_1,x_2)\in\dR^2\colon (x_1,x_2)\in\Omg\} = \{(x_2,x_1)\in\dR^2\colon (x_1,x_2)\in\Omg\}.
\end{equation}
A square and a disk are examples of such domains. One can also construct $C^\infty$-smooth convex
domains satisfying~\eqref{eq:symm} different from the disk.
For example, the domain 
\[
\{(x_1,x_2)\in\dR^2\colon x_1^4+x_2^4 < 1\}
\]
satisfies~\eqref{eq:symm}.
We also immediately get from~\eqref{eq:symm} that 
\[
	\Omg = \{(x_1,-x_2)\in\dR^2\colon (x_1,x_2)\in\Omg\}.
\]	
Combining the symmetries and convexity of $\Omg$ we conclude that the origin belongs to $\Omg$.
For a bounded, convex $C^\infty$-smooth domain $\Omg$ satisfying~\eqref{eq:symm}
we get that the torsion function $v_\Omg$ satisfies
\begin{equation}\label{eq:symmv}
	v_\Omg(x_1,x_2) = v_\Omg(-x_1,x_2) = v_\Omg(x_1,-x_2) = v_\Omg(x_2,x_1)\qquad \text{for all}\,\, (x_1,x_2)\in\Omg.
\end{equation}
Hence, we conclude that
\begin{equation}\label{eq:p12}	
\begin{aligned}
	&\p_1 v_\Omg(0,x_2) = 0\qquad\text{for all}\,\,\, (0,x_2)\in\Omg,\\  
	&\p_2 v_\Omg(x_1,0) = 0\qquad\text{for all}\,\,\, (x_1,0)\in\Omg.  
\end{aligned}	
\end{equation}
Therefore, the gradient of the torsion function vanishes at the origin. Thus, the maximum of the torsion function is attained also at the origin. From the symmetries in~\eqref{eq:symmv}, the properties~\eqref{eq:p12}
and the definition of the torsion function we arrive that 
\[
	\p^2_1 v_\Omg(0) = 	\p^2_2 v_\Omg(0) = -\frac12,\qquad 
	\p_{12}v_\Omg(0) = 0.
\]
We conclude that the Hessian $\cH_\star$ of $v_\Omg$ at the origin satisfies 
\[	
\det\cH_\star = \frac14.
\]
For a bounded, convex, $C^\infty$-smooth domain $\Omg\subset\dR^2$ satisfying the symmetries~\eqref{eq:symm} and being different from a disk
the following holds
\[
	\frac{2\pi\max v_\Omg}{\sqrt{\det\cH_\star}} = 4\pi\max v_\Omg < |\Omg|,
\]
where we used that $\max v_\Omg < \frac{R^2}{4}$ where $R$ is the radius of the disk of the same area as $\Omg$.
By inspecting Proposition~\ref{prop:GOmg} and its proof we come to the conclusion that the supremum in the characterisation of $G(\Omg)$ in Lemma~\ref{lem:FG2}\,(ii) is not attained in the limit $r\arr 0^+$ for this class of convex domains. Furthermore,   by inspecting Proposition~\ref{prop:FOmg},  we find that the upper and lower bounds are both equal to $4\pi$, hence  
\[
	F(\Omg)  = 4\pi
\]
and the infimum in the characterisation of $F(\Omg)$ in Lemma~\ref{lem:FG2}\,(iii) is thus attained in the limit $r\arr 0^+$. As a result, we get that for any bounded, convex $C^\infty$-smooth domain $\Omg\subset\dR^2$ satisfying~\eqref{eq:symm} and being different from the disk we have
\[
	C(\Omg) = \frac{1}{\max v_\Omg} \frac{G(\Omg)}{F(\Omg)} \ge \frac{|\Omg|}{4\pi\max v_\Omg}  > 1. 
\]
Next, we will show that for bounded, convex, planar $C^\infty$-smooth domains satisfying~\eqref{eq:symm} and being different from the disk the bound in Theorem~\ref{thm:main} does not imply the inequality in Conjecture~\ref{conj:FH} for all sufficiently small $b>0$.
Let us assume without loss of generality that $|\Omg| = \pi$.
It will be more convenient to work with an equivalent version of the bound in Theorem~\ref{thm:main}, which appears in~\eqref{eq:almost}
\begin{equation*}
	\mu_1^b(\Omg) \le \frac{1}{4(\max v_\Omg)^2}\frac{G(\Omg)}{F(\Omg)}\mu_1^{4b\max v_\Omg}(\cB_1).
\end{equation*}
It follows from $G(\Omg) \ge |\Omg| = \pi$, $F(\Omg) = 4\pi$, $\max v_\Omg < \frac{1}{4}$ and the monotonicity of the function in Proposition~\ref{prop:disk}\,(iii)
that for all sufficiently small $b > 0$
\[
\begin{aligned}
	\frac{1}{4(\max v_\Omg)^2}\frac{G(\Omg)}{F(\Omg)}\mu_1^{4b\max v_\Omg}(\cB_1)\ge \frac{1}{16(\max v_\Omg)^2}\mu_1^{4b\max v_\Omg}(\cB_1) > \mu_1^b(\cB_1). 
\end{aligned}
\]
Thus, the bound in Theorem~\ref{thm:main} is weaker than the isoperimetric inequality in Conjecture~\ref{conj:FH} for the class of convex domains satisfying~\eqref{eq:symm} and sufficiently small intensities of the magnetic field.
\end{exam}
\section{The case of the ellipse}\label{sec:FH}
In this section we show that the upper bound obtained in Theorem~\ref{thm:main} implies the inequality in Conjecture~\ref{conj:FH}
for $\Omg$ being an ellipse satisfying $b|\Omg| < \pi$. In the argument we exploit the explicit expression of the torsion function for the ellipse.
Recall that for the semi-axes $\aa, \beta >0$ the ellipse 
is defined in the standard way by
\[
	\cE = \cE_{\aa,\beta} := \left\{(x,y)\in\dR^2\colon\frac{x^2}{\alpha^2} + \frac{y^2}{\beta^2} < 1\right\}.
\]	
\begin{thm}\label{thm:ellipse}
	Assume that $b\aa\beta < 1$.   Then the following inequality holds
	\begin{equation}\label{eq:ellipse1}
		\mu_1^b(\cE_{\aa,\beta}) \le \mu_1^b(\cB_{R_{\alpha,\beta}}),\qquad \text{for}\,\,\, R_{\alpha,\beta}= \frac{\sqrt{2}\aa\beta}{\sqrt{\aa^2+\beta^2}}.
	\end{equation}
	As a consequence,  if $\alpha\not=\beta$ there holds 
	\begin{equation}\label{eq:ellipse2}
		\mu_1^b(\cE_{\aa,\beta}) <  \mu_1^b(\cB_{\sqrt{\aa\beta}}).
	\end{equation}
\end{thm}
\begin{proof}
	The torsion function for the ellipse $\cE$ is explicitly given by the expression (see \eg~\cite[Example 1]{FH})
\begin{equation}\label{eq:torsion-E}
v_\cE(x,y) = \frac{1}{\gamma}\left(1-\frac{x^2}{\aa^2}-\frac{y^2}{\beta^2}\right),\qquad\text{where}\,\,
\gamma := \frac{2}{\aa^2}+\frac{2}{\beta^2}.
\end{equation}
In particular, we have $\max v_\cE = \frac{1}{\gamma}$.
The level curve $\cA_r$ in~\eqref{eq:Ar} associated with the ellipse is parameterized by
\[
\cA_r = \left\{x\in\cE\colon v_\cE(x,y) = \frac{1}{\gamma}(1-r^2)\right\} = \left\{(x,y)\in\dR^2\colon \frac{x^2}{\aa^2}+\frac{y^2}{\beta^2} = r^2\right\},\qquad r\in(0,1). \]
In other words, the set $\cA_r$ is the boundary of an ellipse with semi-axes $r\aa$ and $r\beta$. It can be parameterized by
\begin{equation}\label{eq:parametrization}
[0,2\pi) \ni \tt \mapsto (r\aa\cos\tt,r\beta\sin\tt)^\top.
\end{equation} 
The length of the gradient of the torsion function $v_\cE$
is explicitly given by
\[|\nb v_\cE(x,y)| = \frac{2}{\gamma}\sqrt{\frac{x^2}{\aa^4} + \frac{y^2}{\beta^4}}.\]
The constants $F(\cE)$ and $G(\cE)$ defined as in~\eqref{eq:AB} can also be explicitly computed.
Using the parametrization~\eqref{eq:parametrization}
and the above formula for $|\nb v_\cE|$ we find
\[
\begin{aligned}
F(\cE) &= \esinf_{r\in(0,1)} \int_0^{2\pi}
\frac{\gamma\aa\beta}{2}\frac{\sqrt{r^2\aa^2\sin^2\tt + r^2\beta^2\cos^2\tt}}{\sqrt{r^2\beta^2\cos^2\tt  + r^2\aa^2\sin^2\tt}}\dd\tt = \pi \gamma\aa\beta,\\
G(\cE) &= \essup_{r\in(0,1)} \int_0^{2\pi}
\frac{2}{\gamma\aa\beta r^2}\left(r^2\aa^2\sin^2\tt + r^2\beta^2\cos^2\tt\right)\dd \tt = \frac{2\pi}{\gamma\aa\beta}(\aa^2+\beta^2) = \pi\aa\beta.
\end{aligned}
\]
It follows then from the bound in Theorem~\ref{thm:main} that
\[\mu_1^b(\cE) \le  \mu_1^b(\cB_{2/\sqrt{\gamma}}),\]
which is the inequality~\eqref{eq:ellipse1}.   In view of $\frac{2}{\sqrt{\gamma}} < \sqrt{\aa\beta}$ (for $\alpha\not=\beta$),   the inequality~\eqref{eq:ellipse2} follows from~\eqref{eq:ellipse1} combined with the monotonicity stated in Proposition~\ref{prop:disk}\,(ii).	
\end{proof}

The inequality  \eqref{eq:ellipse1} we get for the ellipse can be alternatively derived via a scaling argument.
In the setting of Theorem~\ref{thm:ellipse},  the 
vector potential of the magnetic field \(\bA\) in \eqref{eq:vector_potential} is expressed as follows,
\[ \bA(x,y)=\frac{b\alpha^2\beta^2}{\alpha^2+\beta^2}\begin{pmatrix}
-\frac{y}{\beta^2}\\ 
\frac{x}{\alpha^2}
\end{pmatrix}.\] 
We perform the change of variables  \(x=\frac{\alpha}{\sqrt{\alpha\beta}} X\) and
 \(y=\frac{\beta}{\sqrt{\alpha\beta}} Y\) in the quadratic form in~\eqref{eq:form},  which maps the disk
 \(\cB=\cB_R\) with \(R=\sqrt{\alpha\beta}\)
 onto the ellipse
  $\cE_{\aa,\beta}$.  Then
 we have
\[ 
	\mu_1^b(\cE_{\alpha,\beta})=\inf\left\{ \hat \frq[u]\colon \int_{\cB}|u|^2\dd X\dd Y=1, u\in H^1(\cB)\right\},
\]
and
\[ 
\hat \frq[u]=\alpha\beta\int_{\cB}\left(\Bigl|\Bigl(\frac1{\alpha}\partial_X+\ii \frac{b\alpha\beta Y}{\beta(\alpha^2+\beta^2)} \Bigr)u\Bigr|^2+\Bigl|\Bigl(\frac1{\beta}\partial_Y-\ii \frac{b\alpha\beta X}{\alpha(\alpha^2+\beta^2)}\Bigr)u \Bigr|^2\right)\dd X\dd Y.
\]
We use as a trial function \(u(r,\tt)=f(r)\)
being the normalized radial ground state of
$\sfH_{b}^\cB$.  We get,
\[\begin{aligned}
\hat \frq[u]&=\alpha\beta\int_0^{2\pi}\int_0^R
\Bigl(\frac{\cos^2\theta}{\alpha^2}+\frac{\sin^2\theta}{\beta^2} \Bigr)\left(|f'(r)|^2+\frac{\alpha^2\beta^2 b^2r^2}{(\alpha^2+\beta^2)^2}|f(r)|^2\right)r\dd r\dd \theta\\
&=\pi \frac{\alpha^2+\beta^2}{\alpha\beta}\int_0^R\left(|f'(r)|^2+\frac{\alpha^2\beta^2 b^2r^2}{(\alpha^2+\beta^2)^2}|f(r)|^2\right)r\dd r\\
&=\frac{\alpha^2+\beta^2}{2\alpha\beta}\int_0^{2\pi}\int_0^R\left(|f'(r)|^2+\frac{4\alpha^2\beta^2}{(\alpha^2+\beta^2)^2}\frac{b^2r^2}{4}|f(r)|^2\right)r\dd r\dd\theta,
\end{aligned}\]
where we used that
\[\int_0^{2\pi}\cos^2\tt\dd\theta=\int_0^{2\pi}\sin^2\tt\dd\theta=\pi. \]
Consequently,    we get that 
\begin{equation}\label{eq:scaling-qf}
\hat \frq[u]\leq \frac{\alpha^2+\beta^2}{2\alpha\beta}\mu_1^{\hat b}(\cB_R),
\end{equation}
where
\[ \hat b= \frac{2\alpha\beta}{\alpha^2+\beta^2}b.\]
By scaling,  we have
\[\mu_1^{\hat b}(\cB_R)=\frac{2\alpha \beta }{\alpha^2+\beta^2}\mu_1^b(\cB_{2/\sqrt{\gamma}}), \]
with  \(\gamma\) from \eqref{eq:torsion-E}.  So we infer from \eqref{eq:scaling-qf} that
\[ \mu_1^b(\cE_{\alpha,\beta})\leq \mu_1^b(\cB_{2/\sqrt{\gamma}}),\]
which is the inequality in \eqref{eq:ellipse1}.

\subsection*{Acknowledgement} This work was initiated when the first author (A.K.) visited the Nuclear Physics Institute,  Czech Academy of Sciences,  Prague, Czech Republic.  A.K.  was partially supported by The Chinese University of Hong Kong,  Shenzhen (grant UDF01003322),  and    the Knut and Alice Wallenburg Foundation (grant KAW 2021.0259).
The second author (V.L.) acknowledges the support by 
the grant No.~21-07129S of the Czech Science Foundation (GA\v{C}R).
\newcommand{\etalchar}[1]{$^{#1}$}

\end{document}